\documentclass{amsart}%
\usepackage{amssymb}
\usepackage{amsfonts}
\usepackage{amsmath}
\usepackage{graphicx}%
\providecommand{\U}[1]{\protect\rule{.1in}{.1in}}
\newtheorem{theorem}{Theorem}
\theoremstyle{plain}

\newtheorem{corollary}{Corollary}

\newtheorem{definition}{Definition}

\newtheorem{remark}{Remark}

\numberwithin{equation}{section}
\begin{document}
\title[Majorization in Spaces with a Curved Geometry]{Majorization in Spaces with a Curved Geometry}
\author{Constantin P. Niculescu}
\address{University of Craiova, Department of Mathematics, Street A. I. Cuza 13,
Craiova 200585, Romania}
\email{cniculescu47@yahoo.com}
\author{Ionel Roven\c{t}a}
\address{University of Craiova, Department of Mathematics, Street A. I. Cuza 13,
Craiova 200585, Romania}
\email{roventaionel@yahoo.com}
\thanks{Corresponding author: Constantin P. Niculescu.}
\date{April 3, 2012}

\begin{abstract}
The Hardy-Littlewood-P\'{o}lya majorization theorem is extended to the
framework of some spaces with a curved geometry (such as the global NPC spaces
and the Wasserstein spaces). We also discuss the connection between our
concept of majorization and the subject of Schur convexity.

\end{abstract}
\maketitle

In 1929, G. H. Hardy, J. E. Littlewood and G. P\'{o}lya \cite{HLP1},
\cite{HLP} have proved an important characterization of convex functions in
terms of a partial ordering of vectors $x=(x_{1},...,x_{n})$ in $\mathbb{R}%
^{n}$. In order to state it we need a preparation. We denote by $x^{\downarrow
}$ the vector with the same entries as $x$ but rearranged in decreasing order,%
\[
x_{1}^{\downarrow}\geq\cdots\geq x_{n}^{\downarrow}.
\]
Then $x$ is \emph{weakly} \emph{majorized} by $y$ (abbreviated, $x\prec_{\ast
}y)$ if%
\begin{equation}
\sum_{i\,=\,1}^{k}\,x_{i}^{\downarrow}\leq\sum_{i\,=\,1}^{k}\,y_{i}%
^{\downarrow}\quad\text{for }k=1,...,n \tag{1}\label{1}%
\end{equation}
and $x$ is \emph{majorized} by $y$ (abbreviated, $x\prec y)$ if in addition
\begin{equation}
\sum_{i\,=\,1}^{n}\,x_{i}^{\downarrow}=\sum_{i\,=\,1}^{n}\,y_{i}^{\downarrow
}\,. \tag{2}\label{2}%
\end{equation}

Intuitively, $x\prec y$ means that the components in $x$ are less spread out
than the components in $y$. As is shown in Theorem 1 below, the concept of
majorization admits an order-free characterization based on the notion of
doubly stochastic matrix. Recall that a matrix $A\in\,$M$_{n}(\mathbb{R})$ is
\emph{doubly stochastic} if it has nonnegative entries and each row and each
column sums to unity.

\begin{theorem}
\label{equiv}\emph{(Hardy, Littlewood and P\'{o}lya \cite{HLP1}, Theorem 8).}
Let $x$ and $y$ be two vectors in $\mathbb{R}^{n}$, whose entries belong to an
interval $I.$ Then the following statements are equivalent:

$i)$ $x\prec y;$

$ii)$ There is a doubly stochastic matrix $A=(a_{ij})_{1\leq i,j\leq n}$ such
that $x=Ay;$

$iii)$ The inequality $\sum_{i=1}^{n}f(x_{i})\leq\sum_{i=1}^{n}f(y_{i})$,
holds for every continuous convex function $f:I\rightarrow\mathbb{R}$.
\end{theorem}

The proof of this result is also available in the recent monographs \cite{MOA}
and \cite{NP2006}.

\begin{remark}
M. Tomi\'{c} \cite{T1949} and H. Weyl \cite{W1949} have noticed the following
characterization of weak majorization: $x\prec_{\ast}y$ if and only if
$\sum_{i=1}^{n}f(x_{i})\leq\sum_{i=1}^{n}f(y_{i})$ for every continuous
nondecreasing convex function $f$ defined on an interval containing the
components of $x$ and $y.$ The reader will find the details in \cite{MOA},
Proposition B2, p. 157.
\end{remark}

Nowadays there are known many important applications of majorization to matrix
theory, numerical analysis, probability, combinatorics, quantum mechanics etc.
See \cite{Bh}, \cite{MOA}, \cite{NP2006}, \cite{NCu}, and \cite{Ph}. They were
made possible by the constant growth of the theory, able to uncover the most
diverse situations.

In what follows we will be interested in a simple but basic extension of the
concept of majorization as was mentioned above: the weighted majorization.
Indeed, the entire subject of majorization can be switched from vectors to
Borel probability measures by identifying a vector $x=(x_{1},...,x_{n})$ in
$\mathbb{R}^{n}$ with the discrete measure $\frac{1}{n}\sum_{i=1}^{n}%
\delta_{x_{i}}$ acting on $\mathbb{R}$. By definition,%
\[
\frac{1}{n}\sum_{i=1}^{n}\delta_{x_{i}}\prec\frac{1}{n}\sum_{i=1}^{n}%
\delta_{y_{i}}%
\]
if the conditions (1) and (2) above are fulfilled, and Theorem 1 can be
equally seen as a characterization of this instance of majorization.

Choquet's theory made available a very general framework of majorization by
allowing the comparison of Borel probability measures whose supports are
contained in a compact convex subset of a locally convex separated space. The
highlights of this theory are presented in \cite{Ph} and refer to a concept of
majorization based on condition $iii)$ in Theorem 1 above. Of interest to us
is the particular case of discrete probability measures on the Euclidean space
$\mathbb{R}^{N},$ that admits an alternative approach via condition $ii)$ in
the same Theorem 1. Indeed, in this case one can introduce a relation of the
form
\begin{equation}
\sum_{i=1}^{m}\lambda_{i}\delta_{x_{i}}\prec\sum_{j=1}^{n}\mu_{j}\delta
_{y_{j}} \tag{3}\label{3}%
\end{equation}
by asking the existence of a $m\times n$-dimensional matrix $A=(a_{ij})_{i,j}$
such that%
\begin{gather}
a_{ij}\geq0,\text{ for all }i,j\tag{4}\label{4}\\
\sum_{j=1}^{n}a_{ij}=1,\text{\quad}i=1,...,m\tag{5}\label{5}\\
\mu_{j}=\sum_{i=1}^{m}a_{ij}\lambda_{i}\text{,\quad}j=1,...,n \tag{6}\label{6}%
\end{gather}
and%
\begin{equation}
x_{i}=\sum_{j=1}^{n}a_{ij}y_{j}\text{,\quad}i=1,...,n \tag{7}\label{7}%
\end{equation}

The matrices verifying the conditions (4)\&(5) are called \emph{stochastic on
rows}. When $m=n$ and all weights $\lambda_{i}$ and $\mu_{j}$ are equal, the
condition (6) assures the \emph{stochasticity on columns, }so in that case we
deal with doubly stochastic matrices.

The fact that (3) implies%
\[
\sum_{i=1}^{m}\lambda_{i}f(x_{i})\prec\sum_{j=1}^{n}\mu_{j}f(y_{j}),
\]
for every continuous convex function $f$ defined on a convex set containing
all points $x_{i}$ and $y_{i},$ is covered by a general result due to S.
Sherman \cite{Sh}. See also the paper of J. Borcea \cite{Bor2007} for a nice
proof and important applications.

It is worth noticing that the extended definition of majorization given by (3)
is related, via equality (7), to an optimization problem as follows:
\[
x_{i}=\arg\min_{z\in\mathbb{R}^{N}}\frac{1}{2}\sum_{j=1}^{n}a_{ij}\left\Vert
z-y_{j}\right\Vert ^{2},\text{\quad for }i=1,...,m.
\]

The aim of the present paper is to discuss the analogue of the relation of
majorization (3) within certain classes of spaces with curved geometry. We
will start with the spaces with global nonpositive curvature (abbreviated,
global NPC spaces). The subject of majorization in these spaces was touched in
\cite{N2009} via a different concept of majorization. Central to us here is
the generalization of Theorem 1.

\begin{definition}
\label{def1}A global NPC space is a complete metric space $M=(M,d)$ for which
the following inequality holds true: for each pair of points $x_{0},x_{1}\in
M$ there exists a point $y\in M$ such that for all points $z\in M,$%
\begin{equation}
d^{2}(z,y)\leq\frac{1}{2}d^{2}(z,x_{0})+\frac{1}{2}d^{2}(z,x_{1})-\frac{1}%
{4}d^{2}(x_{0},x_{1}). \tag{8}\label{NPC}%
\end{equation}

\end{definition}

These spaces are also known as the Cat 0 spaces. See \cite{BH99}. In a global
NPC space, each pair of points $x_{0},x_{1}\in M$\ can be connected by a
geodesic (that is, by a rectifiable curve $\gamma:[0,1]\rightarrow M$ such
that the length of $\gamma|_{[s,t]}$ is $d(\gamma(s),\gamma(t))$ for all
$0\leq s\leq t\leq1)$. Moreover, this geodesic is unique.

In a global NPC space, the geodesics play the role of segments. The point $y$
that appears in Definition 1 is the \emph{midpoint} of $x_{0}$ and $x_{1}$ and
has the property
\[
d(x_{0},y)=d(y,x_{1})=\frac{1}{2}d(x_{0},x_{1}).
\]

Every Hilbert space is a global NPC space. Its geodesics are the line segments.

The upper half-plane \textbf{H}$=\left\{  z\in\mathbb{C}:\operatorname{Im}%
z>0\right\}  $, endowed with the Poincar\'{e} metric,%
\[
ds^{2}=\frac{dx^{2}+dy^{2}}{y^{2}},
\]
constitutes another example of a global NPC space. In this case the geodesics
are the semicircles in \textbf{H} perpendicular to the real axis and the
straight vertical lines ending on the real axis.

A Riemannian manifold $(M,g)$ is a global NPC space if and only if it is
complete, simply connected and of nonpositive sectional curvature. Besides
manifolds, other important examples of global NPC spaces are the Bruhat-Tits
buildings (in particular, the trees). See \cite{BH99}. More information on
global NPC spaces is available in \cite{Ba05}, \cite{Jo97}, and \cite{Sturm}.
See also our papers \cite{N2009} and \cite{NR2009}.

\begin{definition}
\label{def2}A set $C\subset M$ is called \emph{convex} if $\gamma
([0,1])\subset C$ for each geodesic $\gamma:[0,1]\rightarrow M$ joining
$\gamma(0),\gamma(1)\in C$.

A function $\varphi:C\rightarrow\mathbb{R}$ is called convex if $C$ is a
convex set and for each geodesic $\gamma:[0,1]\rightarrow C$ the composition
$\varphi\circ\gamma$ is a convex function in the usual sense, that is,
\[
\varphi(\gamma(t))\leq(1-t)\varphi(\gamma(0))+t\varphi(\gamma(1))
\]
for all $t\in\lbrack0,1].$

The function $\varphi$ is called concave if $-\varphi$ is convex.
\end{definition}

The distance function on a global NPC space $M=(M,d)$ verifies not only the
inequality (\ref{NPC}), but also the following stronger version of it,%
\[
d^{2}(z,x_{t})\leq(1-t)d^{2}(z,x_{0})+td^{2}(z,x_{1})-t(1-t)d^{2}(x_{0}%
,x_{1});
\]
here $z\ $is any point in $C$ and $x_{t}$ is any point on the geodesic
$\gamma$ joining $x_{0},x_{1}\in C$. In terms of Definition \ref{def2}, this
shows that all the functions $d^{2}(\cdot,z)$ are uniformly convex. In
particular, they are convex and the balls are convex sets.

In a global NPC space $M=(M,d)$ the distance function $d$ is convex on
$M\times M$ and also convex are the functions $d(\cdot,z).$ See \cite{Sturm},
Corollary 2.5, for details.

Recall that the direct product of metric spaces $M_{i}=(M_{i},d_{i})$
($i=1,...,n)$ is the metric space $M=(M,d_{M})$ defined by $M=%
{\displaystyle\prod\nolimits_{i=1}^{n}}
M_{i}$ and%
\[
d_{M}(x,y)=\left(  \sum_{i=1}^{n}d_{i}(x_{i},y_{i})^{2}\right)  ^{1/2}.
\]
It is a global NPC space if all factors are global NPC spaces.

When $x_{1},...,x_{m},y_{1},...,y_{n}$ are points of a global NPC space $M,$
and $\lambda_{1},...,\lambda_{m}\in\lbrack0,1]$ are weights that sum to 1, we
will define the relation of majorization%
\begin{equation}
\sum_{i=1}^{m}\lambda_{i}\delta_{x_{i}}\prec\sum_{j=1}^{n}\mu_{j}\delta
_{y_{j}} \tag{9}\label{M}%
\end{equation}
by asking the existence of an $m\times n$-dimensional matrix $A=(a_{ij}%
)_{i,j}$ that is stochastic on rows and verifies the following two conditions:%
\begin{equation}
\mu_{j}=\sum_{i=1}^{m}a_{ij}\lambda_{i},\quad j=1,...,n \tag{10}\label{10}%
\end{equation}
and
\begin{equation}
x_{i}=\arg\min_{z\in M}\frac{1}{2}\sum_{j=1}^{n}a_{ij}d^{2}(z,y_{j}%
),\text{\quad}i=1,...,m. \tag{11}\label{11}%
\end{equation}
The existence and uniqueness of the problems of optimization (11) is assured
by the fact that the objective functions are uniformly convex and positive.
See \cite{Jo97}, Section 3.1, or \cite{Sturm}, Proposition 1.7, p. 3.

Notice that the above definition agrees with the usual one in the Euclidean
case. It is also related to the definition of the barycenter of a Borel
probability measure $\mu$ defined on a global NPC space $M$. Precisely, if
$\mu\in\mathcal{P}_{2}(M)$ (the set of those probability measures under which
all functions $d^{2}(\cdot,z)$ are integrable), then its barycenter is defined
by the formula
\[
\operatorname*{bar}(\mu)=\arg\min_{z\in M}\frac{1}{2}\int_{M}d^{2}%
(z,x)d\mu(x).
\]
This definition, due to E. Cartan \cite{Car}, was inspired by Gauss' Least
Squares Method. A larger approach of the notion of barycenter is offered by
the recent paper of Sturm \cite{Sturm}.

The particular case of discrete probability measures $\lambda=\sum_{i=1}%
^{n}\lambda_{i}\delta_{x_{i}}$ is of special interest because the barycenter
of $\lambda$ can be seen as a good analogue for the convex combination (or
weighted mean) $\lambda_{1}x_{1}+\cdots+\lambda_{n}x_{n}.$ Indeed,%
\[
\operatorname*{bar}(\lambda)=\arg\min_{z\in M}\frac{1}{2}\sum_{i=1}^{n}%
\lambda_{i}d^{2}(z,x_{i}),
\]
and the way $\operatorname*{bar}(\lambda)$ provides a mean with nice features
was recently clarified by Lawson and Lim \cite{LL}. As an immediate
consequence one obtains the relation%
\[
\delta_{\operatorname*{bar}(\lambda)}\prec\lambda.
\]

A word of caution when denoting $\operatorname*{bar}(\lambda)$ as $\lambda
_{1}x_{1}+\cdots+\lambda_{n}x_{n}$. Probably a notation like $\lambda_{1}%
x_{1}\boxplus\cdots\boxplus\lambda_{n}x_{n}$ suits better because
$\operatorname*{bar}(\lambda)$ can be far from the usual the arithmetic mean.
Consider for example the case where $M$ is the space $\operatorname*{Sym}%
^{++}(n,\mathbb{R)}$ (of all positively definite matrices with real
coefficients), endowed with the \emph{trace metric},
\[
d_{\operatorname*{trace}}(A,B)=\left(  \sum_{k=1}^{n}\log^{2}\lambda
_{k}\right)  ^{1/2},
\]
where $\lambda_{1},\dots,\lambda_{n}$ are the eigenvalues of $AB^{-1}$. In
this case
\[
\frac{1}{2}A\boxplus\frac{1}{2}B=A^{1/2}(A^{-1/2}BA^{-1/2})^{1/2}A^{1/2},
\]
that is, it coincides with the geometric mean of $A$ and $B.$ See
\cite{Bh2007}, Section 6.3, or \cite{LL2001}, for details.

Since the convex combinations within a global NPC space lack in general the
property of associativity,%
\[
\sum_{i=1}^{n+1}\lambda_{i}x_{i}=(1-\lambda_{n+1})\left(  \sum_{i=1}^{n}%
\frac{\lambda_{i}}{1-\lambda_{n+1}}x_{i}\right)  +\lambda_{n+1}x_{n+1},
\]
the proof of Jensen's inequality is not trivial even in the discrete case.
This explains why this inequality was first stated in this context only in
2001 by J. Jost\emph{ \cite{J1994}}. We recall it here in the formulation of
Eells and Fuglede \cite{EF}, Proposition 12.3, p. 242:

\begin{theorem}
\label{thm2}\emph{(Jensen's Inequality)}. For any lower semicontinuous convex
function $f:M\rightarrow\mathbb{R}$ and any Borel probability measure $\mu
\in\mathcal{P}_{2}(M)$ we have the inequality%
\[
f(\operatorname*{bar}(\mu))\leq\int_{M}f(x)d\mu(x),
\]
provided the right hand side is well-defined.
\end{theorem}

The proof of Eells and Fuglede is based on the following remark concerning
barycenters: If a probability measure $\mu$ is supported by a convex closed
set $K$, then its barycenter $\operatorname*{bar}(q)$ lies in $K$. A
probabilistic approach of Theorem 2 is due to Sturm \cite{Sturm}.

An immediate consequence of Theorem 2 is the following couple of inequalities
that work for any points $z,x_{1},...,x_{n},\,\allowbreak y_{1},...,y_{n}$ in
a global NPC space:%
\[
d^{2}\left(  \frac{1}{n}x_{1}\boxplus\cdots\boxplus\frac{1}{n}x_{n},z\right)
\leq\frac{d^{2}(x_{1},z)+\cdots+d^{2}(x_{n},z)}{n}%
\]
and%
\[
d\left(  \frac{1}{n}x_{1}\boxplus\cdots\boxplus\frac{1}{n}x_{n},\frac{1}%
{n}y_{1}\boxplus\cdots\boxplus\frac{1}{n}y_{n}\right)  \leq\frac{d(x_{1}%
,y_{1})+\cdots+d(x_{n},y_{n})}{n}.
\]

The next theorem offers a partial extension of Hardy-Littlewood-P\'{o}lya
Theorem to the context of global NPC spaces.

\begin{theorem}
\label{thm3}If%
\[
\sum_{i=1}^{m}\lambda_{i}\delta_{x_{i}}\prec\sum_{j=1}^{n}\mu_{j}\delta
_{y_{j}},
\]
in the global NPC space $M,$ then, for every continuous convex function $f$
defined on a convex subset $U\subset M$ containing all points $x_{i}$ and
$y_{j}$ we have
\[
\sum_{i=1}^{m}\lambda_{i}f(x_{i})\leq\sum_{j=1}^{n}\mu_{j}f(y_{j}).
\]

\end{theorem}

\begin{proof}
By our hypothesis, there is an $m\times n$-dimensional matrix $A=(a_{ij}%
)_{i,j}$ that is stochastic on rows and verifies the conditions (10) and (11).
The last condition, shows that each point $x_{i}$ is the barycenter of the
probability measure $\sum_{j=1}^{n}a_{ij}\delta_{y_{j}}$, so by Jensen's
inequality we infer that
\[
f(x_{i})\leq\sum_{j=1}^{n}a_{ij}f(y_{j}).
\]
Multiplying each side by $\lambda_{i}$ and then summing up over $i$ from $1$
to $m,$ we conclude that
\begin{align*}
\sum_{i=1}^{m}\lambda_{i}f(x_{i})  &  \leq\sum_{i=1}^{m}\left(  \lambda
_{i}\sum_{j=1}^{n}a_{ij}f(y_{j})\right) \\
&  =\sum_{j=1}^{n}\left(  \sum_{i=1}^{m}a_{ij}\lambda_{i}\right)  f(y_{j})\\
&  =\sum_{j=1}^{n}\mu_{j}f(y_{j}).
\end{align*}

\end{proof}

In a global NPC space the distance function from a convex set is a convex
function. See \cite{Sturm}, Corollary 2.5. Combining this fact with Theorem
\ref{thm3} we infer the following result.

\begin{corollary}
\label{cor1}If%
\[
\sum_{i=1}^{m}\lambda_{i}\delta_{x_{i}}\prec\sum_{j=1}^{n}\mu_{j}\delta
_{y_{j}},
\]
and all coefficients $\lambda_{i}$ are positive, then $\{x_{1},...,x_{m}\}$ is
contained in the convex hull of $\{y_{1},...,y_{n}\}.$

In particular, the points $x_{i}$ spread out less than the points $y_{j}.$
\end{corollary}

Another application of Theorem \ref{thm3} yields a new set of inequalities
verified by the functions $d(\cdot,$ $z)$ in a global NPC space $M$. These
functions are convex and the same is true for the functions $f(d(\cdot,$ $z))$
whenever $f$ is a continuous nondecreasing convex function defined on
$\mathbb{R}_{+}.$ According to Theorem \ref{thm3}, if $\frac{1}{n}\sum
_{i=1}^{n}\delta_{x_{i}}\prec\frac{1}{n}\sum_{i=1}^{n}\delta_{y_{i}}$ in $M,$
then $\sum_{i=1}^{n}f(d(x_{i},z))\leq\sum_{i=1}^{n}f(d(y_{i},z)).$\ Taking
into account Remark 1 we arrive at the following result:

\begin{corollary}
\label{cor2}If $\frac{1}{n}\sum_{i=1}^{n}\delta_{x_{i}}\prec\frac{1}{n}%
\sum_{i=1}^{n}\delta_{y_{i}}$ in $M=(M,d),$ then for all $z\in M,$%
\[
\left(  d(x_{1},z),...,d(x_{n},z)\right)  \prec_{\ast}\left(  d(y_{1}%
,z),...,d(y_{n},z)\right)
\]

\end{corollary}

According to a result due to Ando (see \cite{MOA}, Theorem B.3a, p. 158), the
converse of Corollary \ref{cor2} works when $M=\mathbb{R}$.

The entropy function,%
\[
H(t)=-t\log t,
\]
is concave and decreasing for $t\in\lbrack1/e,\infty),$ so by Corollary
\ref{cor2} we infer that%
\[%
{\displaystyle\prod\nolimits_{i=1}^{n}}
d(x_{i},z)^{d(x_{i},z)}\geq%
{\displaystyle\prod\nolimits_{i=1}^{n}}
d(y_{i},z)^{d(y_{i},z)},
\]
whenever $\frac{1}{n}\sum_{i=1}^{n}\delta_{x_{i}}\prec\frac{1}{n}\sum
_{j=1}^{n}\delta_{y_{j}}$ and all the points $x_{i}$ and $y_{i}$ are at a
distance $\geq1/e$ from $z.$\newline

Many other inequalities involving distances in a global NPC space can be
derived from Corollary \ref{cor2} and the following result due to Fan and
Mirsky: if $x,y\in\mathbb{R}_{+}^{n},$ then $x\prec_{\ast}y$ if and only if%
\[
\Phi(x)\leq\Phi(y)
\]
for all functions $\Phi:\mathbb{R}^{n}\rightarrow\mathbb{R}$ such that:

\begin{enumerate}
\item $\Phi(x)>0$ when $x\neq0;$

\item $\Phi(\alpha x)=\left\vert \alpha\right\vert \Phi(x)$ for all real
$\alpha;$

\item $\Phi(x+y)\leq\Phi(x)+\Phi(y);$

\item $\Phi(x_{1},...,x_{n})=\Phi(\varepsilon_{1}x_{\pi(1)},...,\varepsilon
_{n}x_{\pi(n)})$ whenever each $\varepsilon_{i}$ belongs to $\{-1,1\}$ and
$\pi$ is any permutation of $\left\{  1,...,n\right\}  .$
\end{enumerate}

For details, see \cite{MOA}, Proposition B6, p. 160.

It is worth noticing the connection between our definition of majorization and
the subject of Schur convexity (as presented in \cite{MOA}):

\begin{theorem}
\label{thm4}Suppose that $\frac{1}{n}\sum_{i=1}^{n}\delta_{x_{i}}\prec\frac
{1}{n}\sum_{i=1}^{n}\delta_{y_{i}}$ in the global NPC space $M=(M,d),$ and
$f:M^{n}\rightarrow\mathbb{R}$ is a continuous convex function invariant under
the permutation of coordinates. Then%
\[
f(x_{1},...,x_{n})\leq f(y_{1},...,y_{n}).
\]

\end{theorem}

\begin{proof}
For the sake of simplicity we will restrict here to the case where $n=3.$

According to the definition of majorization, if $\frac{1}{3}\sum_{i=1}%
^{3}\delta_{x_{i}}\prec\frac{1}{3}\sum_{i=1}^{3}\delta_{y_{i}}$, then there
exists a doubly stochastic matrix $A=(a_{ij})_{i,j=1}^{3}$ such that%
\[
x_{i}=\operatorname*{bar}(\sum_{j=1}^{n}a_{ij}\delta_{y_{j}})\quad\text{for
}i=1,...,n.
\]
As $A$ can be uniquely represented under the form%
\[
A=\left(
\begin{array}
[c]{ccc}%
\lambda_{1}+\lambda_{2} & \lambda_{3}+\lambda_{5} & \lambda_{4}+\lambda_{6}\\
\lambda_{3}+\lambda_{4} & \lambda_{1}+\lambda_{6} & \lambda_{2}+\lambda_{5}\\
\lambda_{5}+\lambda_{6} & \lambda_{2}+\lambda_{4} & \lambda_{1}+\lambda_{3}%
\end{array}
\right)  ,
\]
where all $\lambda_{k}$ are nonnegative and $\sum_{k=1}^{6}\lambda_{k}=1$ (a
simple matter of linear algebra) we can represent the elements $x_{j}$ as
\begin{align*}
x_{1}  &  =\operatorname*{bar}((\lambda_{1}+\lambda_{2})\delta_{y_{1}%
}+(\lambda_{3}+\lambda_{4})\delta_{y_{2}}+(\lambda_{5}+\lambda_{6}%
)\delta_{y_{3}}),\\
x_{2}  &  =\operatorname*{bar}((\lambda_{3}+\lambda_{5})\delta_{y_{1}%
}+(\lambda_{1}+\lambda_{6})\delta_{y_{2}}+(\lambda_{2}+\lambda_{4}%
)\delta_{y_{3}}),\\
x_{3}  &  =\operatorname*{bar}\left(  (\lambda_{4}+\lambda_{6})\delta_{y_{1}%
}+(\lambda_{2}+\lambda_{5})\delta_{y_{2}}+(\lambda_{1}+\lambda_{3}%
)\delta_{y_{3}}\right)  .
\end{align*}

It is easy to see that $(x_{1},x_{2},x_{3})$ is the barycenter of%
\begin{align*}
\mu &  =\lambda_{1}\delta_{(y_{1},y_{2},y_{3})}+\lambda_{2}\delta
_{(y_{1},y_{3},y_{2})}+\lambda_{3}\delta_{(y_{2},y_{1},y_{3})}\\
&  +\lambda_{4}\delta_{(y_{2},y_{3},y_{1})}+\lambda_{5}\delta_{(y_{3}%
,y_{1},y_{2})}+\lambda_{6}\delta_{(y_{3},y_{2},y_{1})},
\end{align*}
so by Jensen's inequality and the symmetry of $f$ we get%
\begin{multline*}
f(x_{1},...,x_{n})\leq\lambda_{1}f(y_{1},y_{2},y_{3})+\lambda_{2}f(y_{1}%
,y_{3},y_{2})+\lambda_{3}f(y_{2},y_{1},y_{3})\\
+\lambda_{4}f(y_{2},y_{3},y_{1})+\lambda_{5}f(y_{3},y_{1},y_{2})+\lambda
_{6}f(y_{3},y_{2},y_{1})\\
=(\lambda_{1}+\cdots+\lambda_{6})f(y_{1},y_{2},y_{3})=f(y_{1},y_{2},y_{3}).
\end{multline*}

\end{proof}

The following consequence of Theorem \ref{thm4} relates the majorization of
measures to the dispersion of their supports.

\begin{corollary}
\label{cor3}If $\frac{1}{n}\sum_{i=1}^{n}\delta_{x_{i}}\prec\frac{1}{n}%
\sum_{i=1}^{n}\delta_{y_{i}}$ in the global NPC space $M=(M,d)$, then%
\[
\sum_{1\leq i<j\leq n}d^{\alpha}(x_{i},x_{j})\leq\sum_{1\leq i<j\leq
n}d^{\alpha}(y_{i},y_{j})
\]
for every $\alpha\geq1$.
\end{corollary}

Alert readers have probably already noticed that essential for the theory of
majorization presented above is the occurrence of the following two facts:

\begin{enumerate}
\item the existence of a unique minimizer for the functionals of the form%
\[
J(x)=\frac{1}{2}\sum_{i=1}^{m}\lambda_{i}d^{2}(x,x_{i})
\]
(thought of as the barycenter $\operatorname*{bar}$($\lambda)$ of the discrete
probability measure $\lambda=\sum_{i=1}^{m}\lambda_{i}\delta_{x_{i}});$

\item the Jensen type inequality,%
\[
f(\operatorname*{bar}(\lambda))\leq\int fd\lambda=\sum_{i=1}^{m}\lambda
_{i}f(x_{i}),
\]
for $f$ in our class of generalized convex functions.
\end{enumerate}

The recent paper of Agueh and Carlier \cite{AC} shows that such a framework is
available also in the case of certain Borel probability measures, equipped
with the Wasserstein metric. More precisely they consider the space
$\mathcal{P}_{2}(\mathbb{R}^{N})$ (of all Borel probability measures on
$\mathbb{R}^{N}$ having finite second moments) endowed with the Wasserstein
metric,
\[
\mathcal{W}_{2}(\mu,\nu)=\inf\left(  \int_{\mathbb{R}^{N}\times\mathbb{R}^{N}%
}\left\Vert x-y\right\Vert ^{2}d\gamma(x,y)\right)  ^{1/2},
\]
where the infimum is taken over all Borel probability measures $\gamma$ on
$\mathbb{R}^{N}\times\mathbb{R}^{N}$ with marginals $\mu$ and $\nu$.

The barycenter $\operatorname*{bar}(\sum_{i=1}^{m}\lambda_{i}\delta_{\nu_{i}%
}),$ of a discrete probability measure $\sum_{i=1}^{m}\lambda_{i}\delta
_{\nu_{i}},$ is defined as the minimizer of the functional%
\[
J(\nu)=\frac{1}{2}\sum_{i=1}^{m}\lambda_{i}\mathcal{W}_{2}^{2}(\nu_{i},\nu).
\]
This minimizer is unique when at least one of the measures $\nu_{i}$ vanishes
on every Borel set of Hausdorff dimension $N-1$. See \cite{AC}, Proposition
2.2 and Proposition 3.5.

The natural class of convex function on the Wasserstein space is that of
functions convex along barycenters. According to \cite{AC}, Definition 7.1, a
function $\mathcal{F}:$ $\mathcal{P}_{2}(\mathbb{R}^{N})\rightarrow\mathbb{R}$
is said to be\emph{ convex along barycenters} if for any discrete probability
measure $\sum_{i=1}^{m}\lambda_{i}\delta_{\nu_{i}}$ on $\mathcal{P}%
_{2}(\mathbb{R}^{N})$ we have%
\[
\mathcal{F}(\operatorname*{bar}(\sum_{i=1}^{m}\lambda_{i}\delta_{\nu_{i}%
}))\leq\sum_{i=1}^{m}\lambda_{i}\mathcal{F}(\nu_{i}).
\]

This notion of convexity coincides with the notion of displacement convexity
introduced by McCann \cite{Mc} if $N=1,$ and is stronger than this in the
general case. However, the main examples of displacement convex functions
(such as the the internal energy, the potential energy and the interaction
energy) are also examples of functions convex along barycenters. See
\cite{AC}, Proposition 7.7.

\begin{theorem}
The concept of majorization and all results noticed in the case of global
NPC\ spaces (in particular, Theorem 3 and Theorem 4) remain valid in the
context discrete probability measures on $\mathcal{P}_{2}(\mathbb{R}^{N})$
having unique barycenters and the functions $\mathcal{F}:$ $\mathcal{P}%
_{2}(\mathbb{R}^{N})\rightarrow\mathbb{R}$ convex along barycenters.
\end{theorem}

We end our paper with an open problem that arises in connection to Rado's
geometric characterization of majorization in $\mathbb{R}^{n}:$ $(x_{1}%
,...,x_{n})\prec(y_{1},...,y_{n})$ in $\mathbb{R}^{n}$ if and only if
$(x_{1},...,x_{n})$ lies in the convex hull of the $n!$ permutations of
$(y_{1},...,y_{n})$. See \cite{MOA}, Corollary B.3, p. 34. A relation of
majorization of this kind can be introduced in the power space $M^{n}$ (of any
global NPC space $M=(M,d)$ as well as of $\mathcal{P}_{2}(\mathbb{R}^{N}))$ by
putting%
\[
(x_{1},...,x_{n})\prec(y_{1},...,y_{n})\text{ }%
\]
if $\frac{1}{n}\sum_{i=1}^{n}\delta_{x_{i}}\prec\frac{1}{n}\sum_{i=1}%
^{n}\delta_{y_{i}}.$ The proof of Theorem \ref{thm4} yields immediately the
necessity part of Rado's characterization: if $(x_{1},...,x_{n})\prec
(y_{1},...,y_{n})$ in $M^{n},$ then $(x_{1},...,x_{n})$ lies in the convex
hull of the $n!$ permutations of $(y_{1},...,y_{n})$. Do the converse work? We
know that the answer is positive if $M$ is a Hilbert space but the general
case remains open.

\medskip

\noindent\textbf{Acknowledgement}. This paper is supported by a grant of the
Romanian National Authority for Scientific Research, CNCS -- UEFISCDI, project
number PN-II-ID-PCE-2011-3-0257.

\end{document}